\documentclass[11pt,english]{amsart}

\usepackage{amsmath,leftidx,amsthm}
\usepackage[all]{xy}
\usepackage{xspace}
\usepackage[psamsfonts]{amssymb}
\usepackage[latin1]{inputenc}
\usepackage{graphicx,color}
\usepackage{hyperref,fancyhdr,dutchcal,appendix}
\usepackage{xcolor}

\newtheorem*{remark}{\bf Remark}

\newtheorem{theorem}{\bf Theorem}
\newtheorem{proposition}[theorem]{\bf Proposition}
\newtheorem{definition}{\bf Definition}
\newtheorem*{mainTheorem}{\bf Main Theorem}

\newtheorem{lemma}[theorem]{\bf Lemma}
\newtheorem{corollary}[theorem]{\bf Corollary}

\newtheorem*{Corollary}{\bf Corollary}

\def\C{{\mathbb C}}

\def\R{{\mathbb R}}
\def\Z{{\mathbb Z}}

\def\p{\mathbb{P}}

\title[Geometric Dynamical Northcott for Plane automorphisms]{The Geometric Dynamical Northcott Property For Regular Polynomial Automorphisms of the Affine Plane}
\author{Thomas Gauthier}
\address{CMLS, Ecole Polytechnique, Institut Polytechnique de Paris, 91128 Palaiseau Cedex, France}
\email{thomas.gauthier@polytechnique.edu}
\author{Gabriel Vigny}
\address{LAMFA, Universit\'e de Picardie Jules Verne, 33 rue Saint-Leu, 80039 AMIENS Cedex 1, FRANCE}
\email{gabriel.vigny@u-picardie.fr}
\thanks{Both authors are partially supported by the ANR grant Fatou ANR-17-CE40-0002-01.}
\thanks{Mathematics~Subject~Classification~(2010):  	37P15, 37P30, 37F45.}

\begin{document}
\begin{abstract}
We establish the finiteness of periodic points, that we called \emph{Geometric Dynamical Northcott Property}, for regular polynomials automorphisms of the affine plane over a function field $\mathbf{K}$ of characteristic zero, improving results of Ingram. 
  
  For that, we show that when $\mathbf{K}$ is the field of rational functions of a smooth complex projective curve, the canonical height of a subvariety is the mass of an appropriate bifurcation current and that a marked point is stable if and only if its canonical height is zero. We then establish the Geometric Dynamical Northcott Property using a similarity argument. 
    \end{abstract}

\maketitle

\section{Introduction}
 Let $k$ be a field of characteristic zero, $\mathcal{B}$ a normal projective $k$-variety, and let $\mathbf{K}:=k(\mathcal{B})$ be its field of rational functions. A regular plane automorphism over the function field $\mathbf{K}$ is a polynomial automorphism of the affine plane $\mathbb{A}^2(\mathbf{K})$ such that the unique indeterminacy point $I^+$ of its extension to $\p^2(\mathbf{K})$ is distinct to the unique indeterminacy point $I^-$ of the extension of $f^{-1}$ to $\p^2(\mathbf{K})$.

Let $h:\p^2(\bar{\mathbf{K}})\to\R_+$ be the standard height function on $\p^2(\bar{\mathbf{K}})$, i.e. the height function $h=h_{\p^2,L}$ associated with the ample linebundle $L:=\mathcal{O}_{\p^2}(1)$. Following  Kawaguchi~\cite{Kawaguchi-Henon} in the number field case, one can define three different canonical heights for $f$:
\[\widehat{h}_f^+:=\lim_{n\to+\infty}d^{-n}h\circ f^n,  \quad \widehat{h}_f^-:=\lim_{n\to+\infty}d^{-n}h\circ f^{-n} \quad \text{and} \quad \widehat{h}_f:=\widehat{h}_f^+ + \widehat{h}_f^-,\]
where $d$ is the common degree of $f$ and $f^{-1}$. 
The height function $\widehat{h}_f^+$ (resp. $\widehat{h}_f^-$) detects the arithmetic complexity of the forward orbit (resp. of the backward orbit) of a point in $\mathbb{A}^2(\bar{\mathbf{K}})$.

A particularly interesting case of regular plane automorphisms is H\'enon maps, i.e. maps of the form  $f(x, y) = (ay, x + p (y))$ with $a\in \mathbf{K}^*$ and $p (x) \in \mathbf{K}[x]$. In that setting, Ingram proved the following (\cite[Theorem 1.2]{Ingram_Henon})
\begin{theorem}[Ingram]\label{tm:Ingram}
  Let $k$ be any field and let $\mathbf{K}$ be the field of rational functions of a smooth projective $k$-variety. Let $f(x, y) = (y, x + p(y))$ for
	$p (x) \in \mathbf{K}[x]$ of degree at least $2$. Then either $f$ is isotrivial or else the set of elements $z \in \mathbb{A}^2 (\mathbf{K})$ with $\widehat{h}_f(z)=0$, is finite, bounded in
	size in terms of the number of places of bad reduction for $f$. In particular,	if $f$ is not isotrivial, then $\widehat{h}_f(z)=0$ if and only if $z$ is periodic for $f$.
\end{theorem}
The map $f$ is isotrivial if, after a suitable change of coordinates, the coefficients of $f$ are constant, i.e. belong to $k$. 
Ingram ask whether one can prove a similar statement for  $f(x, y) = (ay, x + p (y))$ with $a\in \mathbf{K}^*$. This is the purpose of this article in the case where $k$ has characteristic zero (Ingram result allows positive characteristic). 
More precisely, 
\begin{enumerate}
	\item We generalize the above statement to any regular polynomial automorphism $f:\mathbb{A}^2\to\mathbb{A}^2$ defined over a function field of characteristic $0$. More precisely, we establish the \emph{geometric dynamical Northcott Property}: if the map $f$ is not isotrivial, then $\widehat{h}_f(z)=0$ if and only if $z$ is periodic for $f$ and they are only finitely many such points.
	\item We replace the hypothesis that $\widehat{h}_f(P)=0$ with the a priori weaker one $\widehat{h}^+_f(P)=0$.
	\item We express, in the case where $\mathbf{K}$ is the field of rational functions of a smooth complex projective curve, the canonical height of a subvariety as the mass of an appropriate bifurcation current. Then, we show a marked point is stable, in the sense of complex dynamics, if and only if its canonical height is zero.
\end{enumerate}

\medskip

 Consider a regular plane automorphism $f$ over the function field $\mathbf{K}$. As $k$ is a field of characteristic zero, up to replacing it with an algebraic extension, we can assume there exists an algebraically closed subfield $k'$ of $k$ such that $f$ is defined over $k'$ and the transcendental degree of $k'$ over $\overline{\mathbb{Q}}$ is finite. In particular, $k'$ can be embedded in the field $\C(\mathcal{A})$ of rational functions of a normal projective complex variety $\mathcal{A}$, so that we can assume $\mathbf{K}=\C(\mathcal{A}\times\mathcal{B})$. Hence, we can assume $f$ is defined over $\C(\mathcal{B})$ for $\mathcal{B}$ a complex normal projective variety. We thus restrict to the case $k=\C$ in the rest of the paper and $\mathbf{K}=\C(\mathcal{B})$. This will enable the use of complex methods.

To a regular automorphism $f:\mathbb{A}_\mathbf{K}^2\to \mathbb{A}_\mathbf{K}^2$, we can associate a \emph{model} over $\mathcal{B}$ where $\mathbf{K}=\C(\mathcal{B})$, i.e. a birational map $\mathcal{f}:\mathbb{A}^2(\C)\times\mathcal{B}\dashrightarrow(f_\lambda(z),\lambda)\in\mathbb{A}^2(\C)\times\mathcal{B}$ 
such that $\pi\circ \mathcal{f}=\pi$, where $\pi:\mathbb{A}^2(\C) \times \mathcal{B}\to\mathcal{B}$ is the canonical projection, and such that there exists a Zariski open subset $\Lambda\subset \mathcal{B}$ for which $\mathcal{f}$ restricts to $\mathbb{A}^2(\mathbb{C})\times\Lambda$ as an automorphism and such that $f_\lambda:=\mathcal{f}|_{\mathbb{A}^2(\C)\times\{\lambda\}}$ is a complex regular polynomial automorphism for any $\lambda\in\Lambda$. The map $f$ can be identified with the restriction of $\mathcal{f}$ to the generic fiber of $\pi$. The open set $\Lambda$ is the \emph{regular part} of the family $\mathcal{f}$.
To any point $z\in\mathbb{A}^2(\mathbf{K})$ one can also associate a rational map $\mathcal{z}:\mathcal{B}\to\C$ such that $\mathcal{z}$ is defined on $\Lambda$. Such $\mathcal{z}$ is called a \emph{marked point}. We say that $\mathcal{z}$ is \emph{stable} if the sequence of iterates $\lambda \mapsto f_\lambda^n(\mathcal{z}(\lambda))$ is normal on compact subsets of $\Lambda$ (see Remark~\ref{rm:Bishop} below). 

\medskip

Finally, we say that a regular polynomial automorphism $f:\mathbb{A}^2_\mathbf{K}\to\mathbb{A}^2_\mathbf{K}$ is \emph{isotrivial} if there exists an affine automorphism $\varphi:\mathbb{A}^2_{\bar{\mathbf{K}}}\to\mathbb{A}^2_{\bar{\mathbf{K}}}$ such that $\varphi^{-1}\circ f\circ \varphi$ is defined over $\C$, or equivalently if for any model $\mathcal{f}:\C^2\times\mathcal{B}\dashrightarrow\C^2\times\mathcal{B}$ with regular part $\Lambda$ and for any $\lambda,\lambda'\in\Lambda$, there is an affine automorphism $\varphi_{\lambda,\lambda'}:\C^2\to\C^2$ such that $\varphi_{\lambda,\lambda'}^{-1}\circ f_\lambda\circ \varphi_{\lambda,\lambda'}=f_{\lambda'}$.

\medskip

Our result can then be stated as

\begin{mainTheorem}
Let $\mathcal{f}:\mathbb{C}^2\times\mathcal{B}\dashrightarrow\mathbb{C}^2\times\mathcal{B}$ be a non-isotrivial algebraic family of regular polynomial automorphisms of degree $d\geq2$ parametrized by a complex projective variety $\mathcal{B}$ with regular part $\Lambda$ and let $f:\mathbb{A}^2_\mathbf{K}\to\mathbb{A}^2_\mathbf{K}$ be the induced regular automorphism over the field $\mathbf{K}=\C(\mathcal{B})$ of rational functions of $\mathcal{B}$. Then
\begin{enumerate}
\item for any point $z\in\mathbb{A}^2(\mathbf{K})$ with corresponding rational map $\mathcal{z}:\mathcal{B}\dashrightarrow \mathbb{P}^2$,
\begin{equation}
(\mathcal{f},\mathcal{z}) \ \text{is stable} \Longleftrightarrow \widehat{h}_f(z)=0 \Longleftrightarrow \widehat{h}_f^+(z)=0 \Longleftrightarrow z \ \text{is periodic}.\tag{$\star$}\label{equivalence-stable}
\end{equation}
\item the set of marked points $\{\mathcal{z}$ such that $(\mathcal{f},\mathcal{z})$ is stable$\}$ is a finite set. In particular, a stable marked point is stably periodic.
\end{enumerate}
\end{mainTheorem}


This generalizes Ingram's Theorem~\ref{tm:Ingram}:
\begin{Corollary}
Let $k$ be a field of characteristic zero and $\mathbf{K}$ be the field of rational functions of a projective $k$-variety. Let $f:\mathbb{A}^2_\mathbf{K}\to\mathbb{A}^2_\mathbf{K}$ be a regular polynomial automorphism of degree $d\geq2$. Then either $f$ is isotrivial or else the set of elements $P \in \mathbb{A}^2 (\mathbf{K})$ with $\widehat{h}_f(P)=0$, is finite

In particular, if $f$ is not isotrivial, then $\widehat{h}_f(P)=0$ if and only if $P$ is periodic for $f$.
\end{Corollary}
For H\'enon maps over number fields, the finiteness of periodic points is due to Silverman~\cite{Silverman-Henon}. Constructing the canonical heights, Kawaguchi~\cite{Kawaguchi-Henon} proved this result over number fields for regular polynomial automorphisms using the classical Northcott property (see also Lee \cite{Lee-Henon} for an alternate construction). Over a function field, such result was established for polynomials of $\mathbb{A}^1$ by Benedetto \cite{benedetto} and rational maps of $\p^1$ by Baker \cite{Baker-functionfield} and DeMarco \cite{demarco}. In higher dimension, Chatzidakis and Hrushovski gave a model-theoretic version of the statement for polarized endomorphisms in~\cite{Chatzidakis-Hrushovski}. Finally, in \cite{GV_Northcott}, we extended the Northcott property to any polarized endomorphisms, giving a similar statement as that of the Main Theorem.

\medskip
In order to prove the Main Theorem, we adapt the strategy of the proof of \cite[Theorem A]{GV_Northcott} to the case of regular polynomial automorphisms. New difficulties appear since we need to deal with indeterminacy points and saddle periodic points instead of repelling periodic points.

Note that, 
\begin{itemize}
	\item if $f$ is non-isotrivial and $\lambda_0\in\Lambda$ is fixed, then the set of parameters $\lambda$ such that $f_\lambda$  is conjugated to $f_{\lambda_0}$ is a closed subvariety of $\Lambda$,
	\item the restriction of a stable marked point to a subvariety $\mathcal{B}'$ of $\mathcal{B}$ is still stable,
	\item  If $\mathcal{B}'$ is a subvariety of $\mathcal{B}$ and $z \in \C(\mathcal{B})$ has height zero, then it defines a point in $\C(\mathcal{B}')$ whose height is again $0$ by Bézout. Similarly, if for all $\mathcal{B}'$, the corresponding point in $\C(\mathcal{B}')$ has height $0$, then so does $z$.
\end{itemize} 
 In particular, we can reduce to the case where $\dim\mathcal{B}=1$ by a slicing argument. We thus restrict to the case where $\mathbf{K}$ is the field of rational functions of a smooth complex projective curve $\mathcal{B}$. Finally, up to taking a branched cover of $\mathcal{B}$, conjugating by a suitable affine automorphism $\Phi \in \mathrm{Aut}( \mathbb{A}^2_{\bar{\mathbf{K}}})$ and reducing $\Lambda$, we can assume that the indeterminacy point $I(f_\lambda)$ of $f_\lambda$ is $[1:0:0]$ and  the indeterminacy point $I(f^{-1}_\lambda)$ of $f^{-1}_\lambda$ is $[0:1:0]$ for every $\lambda\in\Lambda$.

\medskip

 Pick a point $z\in \mathbb{A}^2(\mathbf{K})$ and let $\mathcal{Z}_n$ be the irreducible subvariety of $\mathbb{P}^2(\C) \times \mathcal{B}$ induced by $f^n(z)$.  Using Kawaguchi's comparison result on heights \cite{Kawaguchi-Henon}, we show that there exists some $B_0>0$, independent of $z$, such that $\widehat{h}_f(z) =0 $ implies $\deg_\mathcal{M}(\mathcal{Z}_n) \leq B_0$ for all $n\in\mathbb{Z}$, so they are only finitely many degrees to consider. 
 
 \medskip
 
 Then, we show that (local) stability is equivalent to having zero height for regular polynomial automorphisms: we express $\widehat{h}^+_f(z)$ in term of an appropriate bifurcation current and show that $\widehat{h}^+_f(z)=0$ if and only if the sequence $(\deg_\mathcal{M}(\mathcal{Z}_n))_{n\geq0}$ is bounded by some constant $D>0$ independent of $n$, hence forward stability is in fact a global notion (Propositions~\ref{prop:global-stable} and \ref{lm:height-henon}). For that, we prove a delicate degeneracy estimate of the Green function to deal with the indeterminacy point, this allows us to construct a DSH cut-off function (DSH functions, introduced by Dinh-Sibony \cite{dinhsibony2}, take into account the complex structure whereas $C^2$ functions do not).  Then we show that if $\mathcal{Z}$ is forward stable, then it is periodic (note that in \cite{Ingram_Henon}, one does not relate zero height with stability).

Finally, we give an application to a conjecture of Kawaguchi and Silverman in the case of regular polynomial automorphism.

\section{Algebraic dynamical pairs of regular polynomial automorphism type}\label{sec:henon}\label{section8}

\subsection{Definition and first properties}

Let $\mathcal{B}$ be a smooth projective complex curve and $\Lambda\subset \mathcal{B}$ a Zariski open subset. We  let $\mathcal{f}:\C^2\times\Lambda\longrightarrow\C^2\times\Lambda$ be an algebraic family of regular polynomial automorphisms of $\C^2$. For each $\lambda\in \Lambda$, $f_\lambda(x,y)=(p_\lambda(x,y),q_\lambda(x,y))$ where $p_\lambda, q_\lambda$ are polynomials in $(x,y)$ that depend holomorphically on $\lambda$ with $\max \deg (p_\lambda, q_\lambda)=d$ independent of $\Lambda$ (up to restricting  $\Lambda$). We assume that the map $f_\lambda$ extends as a birational map $f_\lambda:\p^2\dashrightarrow \p^2$ with
\begin{enumerate}
	\item the only indeterminacy point $I^+$ of $f_\lambda$ is $I^+=[1:0:0]$,
	\item the only indeterminacy point $I^-$ of $f_\lambda^{-1}$ is $I^-=[0:1:0]$.
\end{enumerate}
We call such $\mathcal{f}$ an \emph{algebraic family of regular polynomial automorphisms}. From our normalization, we see that $\deg p_\lambda(x,y)< \deg q_\lambda(x,y) = \deg_y q_\lambda(x,y)$ and that $I^-$ (resp. $I^+$) is a super-attracting fixed point for $f_\lambda$ (resp. $f_\lambda^{-1}$).

A classical example is given by H\'enon maps:
\[\mathcal{f}(x,y,\lambda)= (a(\lambda)y, x+p_\lambda(y), \lambda)  \]
where $\mathcal{p}:\C\times\Lambda\to\C$ is an algebraic family of degree $d>1$ polynomials in one complex variable parametrized by the quasi-projective variety $\Lambda$ with $p_\lambda(y)=\mathcal{p}(y,\lambda)$, $a\in\C[\Lambda]^*$ and the support of $\mathrm{div}(a)$ is contained in the finite set $\mathcal{B}\setminus\Lambda$. 
%
\medskip

We say that the family $\mathcal{f}$ is \emph{isotrivial} if there exist a finite branch cover $\rho:\Lambda'\to\Lambda$ and an algebraic family of invertible affine maps $\psi: \C^2 \times \Lambda' \to\C^2$ such that, writing $\psi((x,y),\lambda):=(\psi_\lambda(x,y),\lambda)$, there exists $\lambda_0\in\Lambda$ such that for all $\lambda\in\Lambda'$,
\[\psi_\lambda^{-1}\circ f_{\rho(\lambda)}\circ \psi_\lambda=f_{\lambda_0}.\]
Let us recall some useful facts on regular polynomial automorphisms (\cite{BedfordSmillie1, BedfordLyubichSmillie, BedfordLyubichSmillie2}). In what follow, a $(p,p)$-current is a current of bidegree $(p,p)$. Such objects are powerful tools in complex dynamics.  
\begin{definition}
The \emph{fibered Green current} of $\mathcal{f}$ is the positive closed $(1,1)$-current $\widehat{T}_\mathcal{f}$ on $\p^2\times\Lambda$ defined by
\[\widehat{T}_\mathcal{f}:=\lim_{n\to+\infty}\frac{1}{d^n}(\mathcal{f}^n)^*(\widehat{\omega}).\]
\end{definition}
It is known that the convergence holds, that $\mathcal{f}^*\widehat{T}_\mathcal{f}:=d\widehat{T}_\mathcal{f}$ and that, for any $\lambda\in\Lambda$, the slice $\widehat{T}_\mathcal{f}|_{\p^2\times \{\lambda\}}$ of $\widehat{T}_\mathcal{f}$ is the forward Green current $T^+_{f_\lambda}$ of $f_\lambda$.

By definition,  $\mathcal{f}^{-1}:\C^2\times\Lambda\to\C^2\times\Lambda$ is also a family of regular polynomial automorphisms so we can similarly defined the backward Green current $\widehat{T}_{\mathcal{f}^{-1}}$ as the fibered Green current of $\mathcal{f}^{-1}$. Then, the current $\widehat{T}_\mathcal{f}\wedge \widehat{T}_{\mathcal{f}^{-1}}$ is well defined and, for any $\lambda\in\Lambda$, the slice $(\widehat{T}_\mathcal{f}\wedge \widehat{T}_{\mathcal{f}^{-1} })|_{\p^2\times\{\lambda\}}$ is the unique maximal entropy measure $\mu_{f_\lambda}$ of $f_\lambda$. As the measure $\mu_{f_\lambda}$ gives no mass to analytic sets and is equidistributed by saddle points, it follows that saddle points are Zariski dense in $\C^2$. So the set of points of the form $((x_0,y_0),\lambda_0)$ such that $(x_0,y_0)$ is a saddle periodic point of $f_{\lambda_0}$ is Zariski dense in $\C^2\times\Lambda$.

Let $\pi:\p^2\times\mathcal{B}\to\mathcal{B}$ be the projection onto the first coordinate.
\begin{definition} Let  $\mathcal{Z}\subset  \p^{2} \times \mathcal{B}$ be an irreducible algebraic curve.  	
	We say that $(\Lambda,\mathcal{f},\mathcal{Z})$ is an \emph{algebraic dynamical pair of regular automorphism-type} if  $\pi|_\mathcal{Z}:\mathcal{Z}\to\mathcal{B}$ is a flat morphism and such that $\mathcal{Z}\cap\left(\p^2\times\Lambda)\right)\subset\C^2\times\Lambda$.\end{definition}

\subsection{A degeneration lemma}

We here prove the following degeneration Lemma in the spirit of \cite[Lemma 12]{GV_Northcott}, which is crucial in what follows. We use ideas of \cite[Lemma 3.2.4.]{ThelinVigny1}. For that, we may view $\Lambda$ as an affine curve: let $F:=\mathcal{B}\setminus\Lambda$ and $D$ be a very ample divisor supported by $F$. This defines an embedding $\iota_1:\mathcal{B}\hookrightarrow\p^{M}$ with $\iota_1^{-1}(\mathbb{A}^{M}(\mathbb{C}))=\Lambda$.
In what follows, $\lambda$ denotes the affine coordinate $\iota(\lambda)$ on $\Lambda$. 
\begin{lemma}\label{goodgrowth2}
	There exist constants $C_1,C_2,C_3>0$ such that for all $n\geq1$, we have
	\[ \frac{1}{d^{n}}(\mathcal{f}^n)^*(\widehat{\omega}) - \widehat{T}_\mathcal{f} = dd^c \phi_n, \ \ \text{on} \ \p^{2}\times\Lambda, \]
with  $|\phi_n(x,y,\lambda)| \leq d^{-n}(C_1 \log^{+}|\lambda|+C_2\log^+ 1/\|d((x,y),I^+)\|+C_3)$, for all $\lambda\in\Lambda$. 
\end{lemma}
 
\begin{proof} Pick any boundary point $\lambda_\star$ of $\Lambda$ in $\mathcal{B}$ and a punctured disk $D^*$ centered at $\lambda_\star$. Let $t$ be a local coordinate in $D^*$ centered at $\lambda_\star$ (i.e. $t=0$ at $\lambda_\star$). We first show that 	
	\begin{equation}\label{eq:step1_henon}
	 \frac{1}{d}\mathcal{f}^*(\widehat{\omega}) - \widehat{\omega} = dd^c \phi_1, \ \ \text{on} \  \p^2\times D^*, 
	 \end{equation}
	with  $|\phi_1(x,y,t)| \leq C_1 \log|t|^{-1}+C_2\log d((x,y),I^+)^{-1}+C_3$, where the constants $C_i$ do not depend on $t\in D^*$.
	
\medskip
	
Observe for that
\[ \frac{1}{d}\mathcal{f}^*(\widehat{\omega}) - \widehat{\omega} =\frac{1}{2d}dd^c\log \frac{|p_t(x,y)|^2+|q_t(x,y)|^2+1}{(|x|^2 +|y|^2+1)^d} \quad \text{on} \  \p^2\times D^*. \]
Let $F_t(x,y,w)=(z^dp_t(x/w,y/w), w^dq_t(x/w,y/w),w^d)$ be the homogeneous lift of $f_t$ to $\C^3$ and write $P:=(x,y,w)$. Then, on $ \p^2\times D^*$, we have 
	\[  \frac{1}{d}\mathcal{f}^*(\widehat{\omega}) - \widehat{\omega} = \frac{1}{2d}dd^c \log \frac{|F_t|^2}{|P|^{2d}}.\]  
The qpsh function $\phi_1:=(2d)^{-1}\log |F_t|^2/ |P|^{2d}$ is well defined on $ \C^3\times D^*$. Take some $\varepsilon>0$ and consider the open set $U:=  \{ 1-\varepsilon<|P|<1+\varepsilon\} \times D^*$. Multiplying by $t^k$ for $k$ large enough and shrinking $D^*$ if necessary, we have
	\[  \phi_1= c \log|t|^{-1} + d^{-1}\log |F_t'(P)| \]
	where  $F_t'(P)$ depends holomorphically on $(P,t)$ and is equal to $(0,0,0)\in \mathbb{C}^{3}$ exactly when $P\in \pi^{-1}(I^+)$ (where $\pi$ denotes the projection $\pi:\C^3\setminus\{0\}\to \p^2$) or when $t=0$. Using  \L ojasiewicz Theorem  \cite[Chapter IV, Proposition p.243]{Loj}, that provides us with two constants $\alpha>0$ and $C>0$ such that on $|P|=1$ we have:
		\[ |F_t'(P)| \geq C (\min (\text{dist}(P,\pi^{-1}(I^+)), \text{dist}(t,0))^{\alpha}.\]
	So reducing $D^*$, we have a constant $\beta$ such that
			\[ \log |F_t'(P)| \geq  \alpha \log |t|  + \alpha \log\text{dist}(P,\pi^{-1}(I^+))+\beta.\]
Since the function $ d^{-1}\log |F_t'(P)| $ is upper semi-continuous, it is bounded from above in $U$ (up to reducing $D^*$). Now from the fact that the projection $\pi:\C^3\setminus\{0\}\to \p^2$ is Lipschitz in $|P|=1$ and the above bound, we get constants $A>0$, $B$, $c$ such that:
\[ c \log|t|^{-1} +B \geq \phi_1 \geq c \log|t|^{-1}+ A\log \text{dist}(\cdot ,I^+).\]
In particular, the assertion \eqref{eq:step1_henon} holds.
	
	\bigskip
	
	Let us now prove the inequality of the lemma in $ \p^2\times D^*$ which will be sufficient by a covering argument (changing the coordinates might change the constants $C_1,C_2,C_3$). Observe first that we can set $\phi_n:= \sum_{k\geq n} \frac{1}{d^k} \phi_1\circ \mathcal{f}^k$.
	So that, summing over $k$, it suffices to check that for any integer $k$, we have
	\[ |\phi_1\circ \mathcal{f}^k |  \leq  (C_1 \log|t|^{-1}+C_2\log d((x,y),I^+)^{-1}+C_3) \]
	where $C_1, C_2$, and $C_3$ do not depend on $k$. For that, we remark first that
	\[|\phi_1(x,y,t)  \circ \mathcal{f}^k | \leq C_1 \log|t|^{-1}+C_2\log d(f_t^k(x,y),I^+)^{-1}+C_3  .\]
Hence, it remains to control the behavior of the term $\log  d(f^k_t(x,y),I^+)^{-1}$. As $I^+$ is a super-attracting fixed point for $f_\lambda^{-1}$, one can find $\kappa>0$ large enough so that:
\[B^+:= \{ (x,y,t), \ |x|\geq \max(|t|^{-\kappa}, |y|)\}\] 
is stable by $\mathcal{f}^{-1}$: $\mathcal{f}^{-1}(B^+) \subset(B^+)$ (again, we shrink $D^*$ if necessary). Furthermore, for $(x,y,t)\in B^+$, we have
\begin{equation}\label{eq:in_B^+}
 d(f^{-1}_t(x,y),I^+) \leq  d((x,y),I^+).  
 \end{equation}
Indeed, this follows from the definition of $B^+$ and the fact that
\[d((x,y),I^+) \simeq \max \left(\frac{1}{|x|}, \frac{|y|}{|x|}\right).\]

Now, take a point $(x,y,t) \in \C^2\times D^*$. Observe that for 
$(x,y,t)\notin B^+ $, then, by definition, either $|x|\leq |t|^{-\kappa}$ or $|x|\leq |y|$. In the former case, $1/|x|\geq |t|^{\kappa}$ so that $d((x,y),I^+)\geq  |t|^{\kappa}$ and in the latter case, $|y|/|x|\geq 1$ so $d((x,y),I^+)\geq  1 \geq |t|^{\kappa}$. In particular, we see that 
if $(x,y,t)\notin B^+$ then $\log d((x,y),I^+)^{-1} \leq \kappa\log |t|^{-1}$. By the above, as $(x,y,t)\notin B^+$, then, $f_t^k(x,y)\notin B^+$ for all $k$ so that $\log d(f_t(x,y),I^+)^{-1} \leq \kappa\log |t|^{-1}$ for all $k$ and the estimate is satisfied in that case.  

On the other hand, for $(x,y,t)\in B^+ $, take $k_0$ the smallest integer such that $f^{k_0}_t(x,y)\notin B^+ $ (one can show that $k_0$ is finite but we will not need that fact). Then, by \eqref{eq:in_B^+}, we have for $k < k_0$, 
 \[d(f^{k}_t(x,y),I^+) \geq  d((x,y),I^+) \] 
 and for $k\geq k_0$:
  \[d(f^{k}_t(x,y),I^+) \geq  |t|^{\kappa}. \]
Thus, for any $k$, we have $\log d(f_t^k(x,y),I^+)^{-1} \leq \kappa \log |t|^{-1} + \log d((x,y), I^+)^{-1}$.

\medskip

The lemma follows, since the local coordinate in any disk $D	^*$ centered at a point of $\mathcal{B}\setminus\Lambda$ grows at a rate $\lambda^{-\tau}$ for some $\tau>0$, where $\lambda$ is the affine coordinate on $\Lambda$ defined above, and since $\mathcal{B}\setminus\Lambda$ is finite.
\end{proof}

\section{Height and stability}

\subsection{Zero height implies uniformly bounded degree}
Let $f:\mathbb{A}^2\to\mathbb{A}^2$ be a regular polynomial automorphism of degree $d$, defined over a field $\mathbf{k}$ of characteristic $0$.
It is known there exists a projective surface $V$ obtained by finitely many blow-ups of points $\pi:V\to\mathbb{P}^2$ such that $f\circ \pi$ and $f^{-1}\circ\pi$ both extend as morphisms $\psi_{\pm}:V\to \mathbb{P}^2$, i.e. the diagram
\begin{equation}\label{eq:blowup}
\xymatrix{ &&\ar[lld]_{\psi_+} V\ar[d]^\pi \ar[rrd]^{\psi_-} &&\\
	\mathbb{P}^2 \ar@{<--}[rr]_{f}  & & \mathbb{P}^2\ar@{-->}[rr]_{f^{-1}}  & & \mathbb{P}^2}
\end{equation}
commutes.
We rely on the next result of Kawaguchi~\cite[Theorem~2.1]{Kawaguchi-Henon}:

\begin{theorem}[Kawaguchi~\cite{Kawaguchi-Henon}]\label{tm:Kawaguchi}
	Let $f:\mathbb{A}^2\to\mathbb{A}^2$ be a regular polynomial automorphism of degree $d\geq2$,  let $V$, $\pi$ and $\psi_{\pm}$ be as in \eqref{eq:blowup}, and let $H_\infty$ be the line at infinity of $\mathbb{A}^2$. Then, as a $\mathbb{Q}$-divisor on $V$,
	\[D:=\psi_+^*H_\infty+\psi_-^*H_\infty-\left(d+\frac{1}{d}\right)\pi^*H_\infty\]
	is effective.
\end{theorem}

As in the case of number fields, this allows to prove that $\widehat{h}_f$ is comparable (from below) with the standard height function $h$ using the functoriality properties of height functions, see~\cite[Theorem~2.3 \& Theorem~4.1]{Kawaguchi-Henon}. We follow Kawaguchi's arguments and claim no originality in the proof. 

\begin{corollary}\label{cor-goodineq}
	Let $f:\mathbb{A}^2\to\mathbb{A}^2$ be a regular polynomial automorphism of degree $d\geq2$, defined over a global function field of characteristic $0$. Then, there is a constant $C>0$ such that for all $z\in\mathbb{A}^2(\bar{\mathbf{K}})$, we have
	\[\widehat{h}_f(z)\geq h(z)-C~.\]
\end{corollary}

\begin{proof}
	We follow Kawaguchi's proof: By definition of the $\mathbb{Q}$-divisor $D$, we have
	\[h_{V,D}=h_{V,\psi_+^*H_\infty}+h_{V,\psi-^*H_\infty}-\left(d+\frac{1}{d}\right)h_{V,\pi^*H_\infty}+O(1).\]
	According to Theorem~\ref{tm:Kawaguchi}, $D$ is effective. Since $\pi(\mathrm{supp}(D))\subset\mathrm{supp}(H_\infty)$, we have 
	$h_{V,D}\geq O(1)$ on $\mathbb{A}^2(\bar{\mathbf{K}})$.
	Pick now $z\in\mathbb{A}^2(\bar{\mathbf{K}})$. As $\pi$ restricts as an isomorphism on $Y:=\pi^{-1}(\mathbb{A}^2)$, there is a unique $\tilde{z}\in Y(\bar{\mathbf{K}})$ with $\pi(\tilde{z})=z$ and the definition of $\psi_{\pm}$ gives
	\[h_{V,\psi_{\pm}^*H_\infty}(\tilde{z})=h_{\mathbb{P}^2,H_\infty}(\psi_{\pm}(\tilde{z}))+O(1)=h(f^{\pm 1}(z))+O(1).\]
	Similarly, we have $h_{V,\pi^*H_\infty}(\tilde{z})=h(z)+O(1)$. To summarize, we proved
	\begin{equation*}
	h(f(z))+h(f^{-1}(z))-\left(d+\frac{1}{d}\right)h(z)\geq -C,
	\end{equation*}
	for some $C>0$ independent of $z$. In particular, for any $z\in\mathbb{A}^2(\bar{\mathbf{K}})$, we have
	\begin{equation}
	\frac{1}{d}h(f(z))+\frac{1}{d}h(f^{-1}(z))-\left(1+\frac{1}{d^2}\right)h(z)\geq -C'\label{eq-goodineq}
	\end{equation}
	for some $C'>0$ independent of $z$. Set $h':=h-\frac{d^2}{(d-1)^2}C'$ in~\eqref{eq-goodineq} so that it becomes:
	\begin{equation*}
	\frac{1}{d}h'(f(z))+\frac{1}{d}h'(f^{-1}(z))\geq \left(1+ \frac{1}{d^2}\right)h(z).
	\end{equation*}
	Applying  to $f(z)$ yields
	\[\frac{1}{d^2}h'(f^2(z))+\frac{1}{d^2}h'(z)\geq \frac{1}{d} \left(1+\frac{1}{d^2}\right)h'(f(z)).\]
	Exchanging the roles of $f$ and $f^{-1}$, we have the same inequality for $f^{-1}$ and we deduce 
	\[\frac{1}{d^2}h'(f^2(z))+\frac{1}{d^2}h'(f^{-2}(z))\geq \left(1+\frac{1}{d^4}\right)h'(z).\]
	An easy induction gives
	\[\frac{1}{d^{2^k}}h'(f^{2^k}(z))+\frac{1}{d^2}h'(f^{-2^k}(z))\geq \left(1+\frac{1}{d^{2^{k+1}}}\right)h'(z).\]
	Making $k\to +\infty$, we find $\widehat{h}_f\geq h'$, as expected.
\end{proof}

\subsection{Stability, bifurcation measure and the canonical height.}
Let us come back to the case where $\mathbf{K}=\C(\mathcal{B})$ is the field of rational functions over a smooth complex projective curve.  We keep conventions and notations of Section~\ref{sec:henon}.

Recall that we defined an embedding $\iota_1:\mathcal{B}\hookrightarrow\p^{M}$ with $\iota_1^{-1}(\mathbb{A}^{M}(\mathbb{C}))=\Lambda$. Let $\omega_{\mathrm{FS},M}$ denote the Fubini-Study form on $\p^{M}$ and $\omega_\mathcal{B}:=\iota_1^*(\omega_{\mathrm{FS},M})$.
Let also $\pi_{\p^2}: \p^2\times \mathcal{B}\to \p^2$ be the  canonical projections, $\omega_{\mathrm{FS},2}$ the Fubini-Study form on $\p^2$ and $\widehat{\omega}:=\pi_{\p^2}^*(\omega_{\mathrm{FS},2})$. Note that the closed positive $(1,1)$-form
\[\widehat{\omega} +\pi^*\omega_\mathcal{B}\]
is a K\"ahler form on $\p^2\times\mathcal{B}$, which is cohomologous to the ample line bundle 
$\mathcal{M}:=\mathcal{L}\otimes\mathcal{N}$ on $\p^2\times\mathcal{B}$,
where $\mathcal{L}:=\pi_{\p^2}^*\mathcal{O}_{\p^2}(1)$ and $\mathcal{N}:=\iota_1^*\mathcal{O}_{\p^N}(1)$.
For any closed positive $(p,p)$-current $S$ on $\p^2\times\mathcal{B}$, and any Borel subset $\Omega$ of $\p^2\times\mathcal{B}$, we set
\[\|S\|_\Omega:=\int_\Omega S\wedge \left(\widehat{\omega}+\pi^*\omega_\mathcal{B}\right)^{3-p}.\]

\begin{definition}
	Consider $(\Lambda,\mathcal{f},\mathcal{Z})$ an algebraic dynamical pair of regular automorphism-type. We say that $(\Lambda,\mathcal{f},\mathcal{Z})$ is \emph{stable} if for any compact subset $K\Subset \p^2\times\Lambda$, we have $\|(\mathcal{f}^{n})_*[\mathcal{Z}]\|_K=O(1)$ as $n\to+\infty$.
\end{definition}
\begin{remark}\label{rm:Bishop}\normalfont
	Let us emphasize that $(\mathcal{f}^n)_*[\mathcal{Z}]=[\mathcal{f}^n(\mathcal{Z})]$, so that, when $(\Lambda,\mathcal{f},\mathcal{Z})$ is stable, the sequence $(\mathcal{f}^n(\mathcal{Z}))_n$ of analytic subsets has locally uniformly bounded mass and, by Bishop Theorem, converges up to extraction to an analytic subset $\mathcal{Z}_\infty$ of $\C^2\times\Lambda$. If in addition $\mathcal{Z}$ is the graph of a morphism $a:\Lambda\to\C^2$, then $(\Lambda,\mathcal{f},\mathcal{Z})$ is stable if and only if the sequence $(\lambda\mapsto f_\lambda^n(a(\lambda)))_n$ converges locally uniformly to a map $a_\infty:\Lambda\to\C^2$. This justifies the definition of stability.
\end{remark}

We now characterize stability in terms of a measure on the parameter space $\Lambda$. To do so, we let
\begin{definition}
	Let $(\Lambda,\mathcal{f},\mathcal{Z})$ be an algebraic dynamical pair of regular automorphism-type. The \emph{bifurcation measure} $\mu_{\mathcal{f},[\mathcal{Z}]}$ of $(\Lambda,\mathcal{f},\mathcal{Z})$ is the positive measure on $\Lambda$
	\[\mu_{\mathcal{f},[\mathcal{Z}]}:=\pi_*\left(\widehat{T}_\mathcal{f}\wedge[\mathcal{Z}]\right).\]
\end{definition}
Remark that, since $\widehat{T}_\mathcal{f}$ has continuous potentials on $(\p^2\setminus\{I^+\})\times \Lambda$, the measure $\mu_{\mathcal{f},[\mathcal{Z}]}$ is well-defined. 

~ 

We now come to the proof of the following, which says that stability is equivalent to having bounded degree under iteration. Before starting the proof, we recall that, since $\mathcal{Z}$ has relative dimension $0$, for any integer $n\in\mathbb{Z}$, we have
\[(\mathcal{f}^n)_*[\mathcal{Z}]=[\mathcal{f}^n(\mathcal{Z})]\]
and that, by definition of $\deg_\mathcal{M}$ and of the mass of a current, we have
\[\|[\mathcal{f}^n(\mathcal{Z})]\|_{\C^2\times\Lambda}=\|[\mathcal{f}^n(\mathcal{Z})]\|_{\p^2(\C)\times\mathcal{B}}=\deg_\mathcal{M}(\mathcal{f}^n(\mathcal{Z})).\]

\begin{proposition}\label{Stable=bounded-henon}
Let $(\Lambda,\mathcal{f},\mathcal{Z})$ be an algebraic dynamical pair of regular automorphism-type. There exists a constant $B>0$ depending only on $(\Lambda,\mathcal{f},\mathcal{Z})$ such that for any $n\geq1$,
\[\left|\deg_\mathcal{M}(\mathcal{f}^n(\mathcal{Z}))-d^n\int_\Lambda\mu_{\mathcal{f},[\mathcal{Z}]}\right|\leq B.\]
\end{proposition}

\begin{proof}
We adopt the same strategy as in \cite{GV_Northcott} using Lemma~\ref{goodgrowth2} to deal with the indeterminacy set. 
For any $A>0$, we pick the following test function 
\[
\Psi_A(\lambda):=  \frac{ \log \max (\|\lambda\|,e^A)- \log \max (\|\lambda\|,e^{2A})}{A} ,\quad\text{for all} \ \lambda\in\Lambda.
\]
Then, $\Psi_A$ is continuous and DSH on $\Lambda$, i.e. $dd^c \Psi_A=T_A^+-T_A^-$
where $T^\pm_A$ are some positive closed $(1,1)$-currents whose masses are finite with $\|T^\pm_A \|\leq C'/A$ for some $C'>0$ depending neither on $A$ nor on $T^\pm_A$. Observe also that $\Psi_A$ is equal to $-1$ in $B(0,e^A)$, and $0$ outside $B(0,e^{2A})$. Since $ \pi\circ \mathcal{f}^{n}= \pi$, we have
\begin{align*}
I^A_n & := \left\langle\frac{1}{d^{n}}(\mathcal{f}^{n})_*[\mathcal{Z}]\wedge\left(\widehat{T}_\mathcal{f}-\widehat{\omega} \right),\Psi_A\circ \pi \right\rangle\\
& =\left\langle [\mathcal{Z}]\wedge\frac{1}{d^{n}}(\mathcal{f}^{n})^{*}\left(\widehat{T}_\mathcal{f}-\widehat{\omega}\right),\Psi_A\circ \pi\circ \mathcal{f}^{n}\right\rangle\\
& = \left\langle [\mathcal{Z}]\wedge\left( \widehat{T}_\mathcal{f}-\frac{1}{d^n}(\mathcal{f}^n)^*\left(\widehat{\omega}\right)\right),\Psi_A\circ \pi\right\rangle,
\end{align*}
where we used $(\mathcal{f}^n)^*\widehat{T}_\mathcal{f}=d^n\cdot \widehat{T}_\mathcal{f}$. By Stokes formula  and using the notations of Lemma~\ref{goodgrowth2}.
\begin{align*}
I^A_n & = \left\langle \phi_n \cdot [\mathcal{Z}],dd^{c}(\Psi_A\circ \pi)\right\rangle.
\end{align*}
 In particular, by the properties of $T^\pm_A$, there exists a constant $B>0$ such that by B\'ezout:
 \begin{align*}
|I^A_n|& \leq \frac{B \|[\mathcal{Z}\|_{\C^2\times\Lambda}}{A d^n}\sup_{   \pi^{-1}(B(0,e^{2A}) \cap \mathcal{Z})}|\phi_n|.
\end{align*}
We now use Lemma \ref{goodgrowth2}. As $\mathcal{C}$ is an algebraic curve, when $A$ is large, we have that 
\[\sup_{\pi^{-1}(B(0,e^{2A}) \cap\mathcal{Z})} d((x,y),I^+)\geq |A|^{-\zeta} \]
for a constant $\zeta$ that does not depend on $A$ (though $\zeta$ depends on $\mathcal{Z}$). This  implies $|I^A_n|\leq B d^{-n}$ where $B$ is a constant that depends on $(\Lambda,\mathcal{f},\mathcal{Z})$ but neither on $A$ nor on $n$. We make $A\to\infty$ and multiply by $d^n$ and find
\[\left|\int_{\C^2\times\Lambda}(\mathcal{f}^{n})_*[\mathcal{Z}]\wedge\widehat{\omega}-d^{n}\int_{\Lambda} \pi_*\left([\mathcal{Z}]\wedge \widehat{T}_\mathcal{f}\right)\right|\leq B .\]
On the other hand, as $ \pi_\Lambda\circ \mathcal{f}^{n}= \pi_\Lambda$ for all $n$, we see that
\begin{align*}
\int_{\C^2\times\Lambda}(\mathcal{f}^{n})_*[\mathcal{Z}]\wedge (\pi^*\omega_\mathcal{B})  =\int_{\C^2\times\Lambda}[\mathcal{Z}]\wedge (\pi^*\omega_\mathcal{B})
 \leq  B' \|[\mathcal{Z}]\|_{\C^2\times\Lambda}
\end{align*}
where $B'$ is a constant that does not depend on $n$. This ends the proof.
\end{proof}

This global estimate of mass allows us to give the following different global characterizations of stability.   
\begin{proposition}\label{prop:global-stable}
Let $(\Lambda,f,\mathcal{Z})$ be an algebraic dynamical pair of regular polynomial automorphism-type.
The following assertions are equivalent:
\begin{enumerate}
\item $(\Lambda,\mathcal{f},\mathcal{Z})$ is stable,
\item the sequence $(\deg_{\mathcal{M}}(\mathcal{f}^n(\mathcal{Z})))_{n\geq1}$ is bounded, 
\item the function $G^+_\mathcal{Z}$ is constant on $\Lambda$ where
\[G_\mathcal{Z}^+:\lambda\in\Lambda\longmapsto \sum_{(x,y)\in\mathcal{Z}\cap\C^2\times\{\lambda\}}G^+_\lambda(x,y)\in\R_+.\]
\end{enumerate}
\end{proposition}

\begin{proof}
We may view $\Lambda$ as an affine curve. According to Proposition~\ref{Stable=bounded-henon}, we have $\mu_{\mathcal{f},[\mathcal{Z}]}=0$ if and only if the sequence $(\deg_{\mathcal{M}}(\mathcal{f}^n(\mathcal{Z})))_{n\geq1}$ is bounded, so $1\iff 2$.

The implication $3\Rightarrow2$ follows from the fact that  $\mu_{f,[\mathcal{Z}]}=dd^cG^+_\mathcal{Z}$. Indeed, the current $\widehat{T}_\mathcal{f}$ satisfies $\widehat{T}_\mathcal{f}=dd^cG^+$ on $\C^2\times\Lambda$ so that
\[dd^cG^+_\mathcal{Z}=\widehat{T}_\mathcal{f}\wedge [\mathcal{Z}].\]
Finally, if $1$ holds, the function $G_\mathcal{Z}$ is harmonic on $\Lambda$. Applying Lemma~\ref{goodgrowth2} for $n=1$ gives
\[0\leq G^+_\lambda(x,y) \leq C_1\log^+|\lambda|+C_2\log^+\|(x,y)\|+C_3,\]
for all $((x,y),\lambda)\in \C^2\times \Lambda$. As $\mathcal{Z}$ is an algebraic curve of $\C^2\times\Lambda$ such that $C_\lambda:=\mathcal{Z}\cap(\C^2\times\{\lambda\})$ is finite for all $\lambda$, this implies 
\[0\leq G^+_\mathcal{Z}(\lambda)\leq C\log^+|\lambda|+C', \quad \lambda\in\Lambda.\]
Recall that $\mathcal{B}$ is a smooth compactification of $\Lambda$. Pick a branch at infinity $\mathfrak{c}\in\mathcal{B}\setminus\Lambda$ of $\Lambda$. If $t$ is a local coordinate at $\mathfrak{c}$. Then there is a constant $c_\mathfrak{c}\geq0$ such that $G_\mathcal{Z}^+(t)=c_\mathfrak{c}\log|t|^{-1}+o(\log|t|^{-1})$ and Stokes Theorem implies
\[\sum_{\mathfrak{c}\in\mathcal{B}\setminus\Lambda}c_\mathfrak{c}=\int_\mathcal{B}dd^cG^+_\mathcal{Z}=\int_{\p^2\times\mathcal{B}}\widehat{T}_\mathcal{f}\wedge [\mathcal{Z}]=0.\]
Coming back to a local parametrization of $\Lambda$ at some $\mathfrak{c}\in\mathcal{B}\setminus\Lambda$, we have $G^+_\mathcal{Z}(t)=o(\log|t|^{-1})$. Define now
\[u(\lambda):=-G^+_\mathcal{Z}(\lambda), \quad \lambda\in\Lambda.\]
The above implies that the function $u$ is subharmonic on $\Lambda$ with $u\leq0$ and whose usc extension to $\mathcal{B}$ is still subharmonic. By the maximum principle, it has to be constant.
\end{proof}

\subsection{Function field versus family}
Recall that, as explained in the introduction, a family (resp. non-isotrivial  family) $\mathcal{f}:\C^2\times\Lambda\to\C^2\times \Lambda$ of regular polynomial automorphisms as above can be associated with a dynamical system over the function field $\mathbf{K}=\C(\mathcal{B})$: it induces a regular polynomial automorphism $f:\mathbb{A}^2_\mathbf{K}\to\mathbb{A}^2_\mathbf{K}$ (resp. a non isotrivial regular polynomial automorphism $f:\mathbb{A}^2_\mathbf{K}\to\mathbb{A}^2_\mathbf{K}$) which is given by
\[f(x,y)=(p(x,y),q(x,y)),\]
where $p,q\in\mathbf{K}[x,y]$.

Moreover, if $(\Lambda,f,\mathcal{Z})$ is a dynamical pair, the curve $\mathcal{Z}$ corresponds to a finite algebraic subvariety $Z$ of $\mathbb{A}^2(\bar{\mathbf{K}})$. Similarly, if $Z$ is a finite algebraic subvariety of $\mathbb{A}^2(\mathbf{K})$ which is defined over $\mathbf{K}$, let $\mathcal{Z}$ be the Zariski closure of $Z$ in $\p^2(\C)\times\mathcal{B}$. This is a curve $\mathcal{Z}\subset\p^2(\C)\times\mathcal{B}$ which is flat over $\mathcal{B}$. Finally, to $z\in\mathbb{A}^2(\mathbf{K})$, one can also associate a marked point $\mathcal{z}:\mathcal{B}\to\C$ such that $\mathcal{z}$ is defined on $\Lambda$, in which case the subvariety $\mathcal{Z}$ corresponds to the closure of the graph of $\mathcal{z}$ restricted to $\Lambda$.  

As in the case of endomorphisms in \cite{GV_Northcott}, we can express the canonical height of $Z$ as the mass of the bifurcation measure.
\begin{proposition}\label{lm:height-henon}
Let $\mathbf{K}:=\C(\mathcal{B})$ be the field of rational functions of a smooth complex projective cuvre and let $(\Lambda,f,\mathcal{Z})$ be an algebraic dynamical pair of regular polynomial automorphism-typ with $\Lambda\subset\mathcal{B}$ and $f:\mathbb{A}^2_\mathbf{K}\to\mathbb{A}^2_\mathbf{K}$ be the induced regular polynomial automorphism and let $Z$ be the finite algebraic subvariety  of $\mathbb{A}^2(\mathbf{K})$ induced by $\mathcal{Z}$. Then
	\[\widehat{h}_f^+(Z)=\int_{\C^2\times\Lambda}[\mathcal{Z}]\wedge \widehat{T}_\mathcal{f} \quad \text{and} \quad \widehat{h}_f^-(Z)=\int_{\C^2\times\Lambda}[\mathcal{Z}]\wedge \widehat{T}_{\mathcal{f}^{-1}}.\]
	In particular, $\widehat{h}_f^+(Z)=0$ if and only if $(\Lambda,\mathcal{f},\mathcal{Z})$ is stable. 
	
	Furthermore, there exists a constant $B_0\geq0$ depending only on $f$ such that, if $z\in\mathbb{A}^2(\mathbf{K})$ corresponds to the marked point $\mathcal{z}:\mathcal{B}\dashrightarrow\mathbb{P}^2$ and $\mathcal{Z}$ is the closure of the graph of $\mathcal{z}$ restricted to $\Lambda$, then 
\begin{center}
$\hat{h}_f(z)=0 $ if and only if for any $n\in\mathbb{Z}$, we have $\deg_{\mathcal{M}}(\mathcal{f}^n(\mathcal{Z}))\leq B_0$.
\end{center}\end{proposition}
\begin{proof}
	The expression of the height is an application of Proposition~\ref{Stable=bounded-henon}. Indeed, for any $n\geq1$, we have
	\[\frac{1}{d^n}h(f^n(Z))=\frac{1}{d^n}\int_{\C^2\times\Lambda}[\mathcal{f}^n(\mathcal{Z})]\wedge\widehat{\omega}=\int_{\C^2\times\Lambda}[\mathcal{Z}]\wedge \widehat{T}_\mathcal{f}+O(d^{-n})\]
	by Proposition~\ref{Stable=bounded-henon}. We conclude letting $n\to+\infty$. Applying everything for $f^{-1}$ instead of $f$ gives the expected formulae.
	The equivalence between height zero and stability follows from Proposition~\ref{prop:global-stable}.
	
	Now, apply Corollary~\ref{cor-goodineq} to $\mathbf{K}=\C(\mathbb{B})$ and $z\in \mathbb{A}^2(\mathbf{K})$:
	\[h(z)\leq \hat{h}_f(z)+C,\]
	where $C$ does not depend on $z$. In particular, if $ \hat{h}_f(z)=0$, we deduce for all $n \in \Z$
		\[h(f^n(z))\leq \hat{h}_f(f^n(z))+C = C.\] 
	Let $\mathcal{N}$ be an ample linebundle on $\mathcal{B}$, be such that $\mathcal{M}:=\pi_1^*\mathcal{O}_{\mathbb{P}^1}(1)\otimes\pi_2^*(\mathcal{N})$ is ample on $\mathbb{P}^2\times\mathcal{B}$. By construction, for any $z\in\mathbb{A}^2(\mathbf{K})$ with Zariski closure $\mathcal{Z}$ in $\mathbb{P}^2\times\mathcal{B}$, we have
	\[h(z)=\left(\mathcal{Z}\cdot c_1(\pi_1^*\mathcal{O}_{\mathbb{P}^2}(1)\right)=\deg_{\mathcal{M}}(\mathcal{Z})-\left(\mathcal{Z}\cdot c_1(\pi_2^*\mathcal{N})\right)=\deg_{\mathcal{M}}(\mathcal{Z})-1.\]
	 If $\mathcal{Z}_n$ is the Zariski closure of $f^n(z)$ for all $n\in \Z$, we find
	\[\deg_{\mathcal{M}}(\mathcal{Z}_n)\leq C+1, \quad \text{for all} \ n\in\mathbb{Z}.\]
	Taking $B_0=C+1$ ends the proof.
\end{proof}

\section{The Geometric Dynamical Northcott Property}

\subsection{Stability implies periodicity and the Main Theorem}

We prove the following rigidity theorem which is concerned with stable dynamical pairs of regular polynomial automorphism type, and then deduce the Main Theorem from it:

\begin{theorem}\label{tm:henon}
Let $(\Lambda,\mathcal{f},\mathcal{Z})$ be a non-isotrivial algebraic dynamical pair of regular polynomial automorphism-type, where $\mathcal{Z}$ is the graph of a marked point $\Lambda\to\C^2$. The following assertions are equivalent:
\begin{enumerate}
\item the pair $(\Lambda,\mathcal{f},\mathcal{Z})$ is stable,
\item the pair $(\Lambda,\mathcal{f}^{-1},\mathcal{Z})$ is stable,
\item  there exists $n>0$ such that $\mathcal{f}^{n}(\mathcal{Z})=\mathcal{Z}$.
\end{enumerate}
\end{theorem}

\begin{proof}
Assume first that both pairs $(\Lambda,\mathcal{f},\mathcal{Z})$ and $(\Lambda,\mathcal{f}^{-1},\mathcal{Z})$ are stable and let us show that $\mathcal{Z}$ is periodic. Assume $\mathcal{Z}$ is not periodic. As $\mathcal{f}$ is a family of regular polynomial automorphisms, $\mathcal{Z}$ is not periodic and $\deg_\mathcal{M}(\mathcal{f}^{n}(\mathcal{Z}))\leq B$ for all $n\in \Z$ for a given constant $B\geq1$, by Proposition~\ref{prop:global-stable} applied to $\mathcal{f}$ and $\mathcal{f}^{-1}$. Let $\mathcal{Z}_B$ be the set of algebraic curves $\mathcal{V}\subset \p^2\times\mathcal{B}$ with $\deg_\mathcal{M}(\mathcal{V})\leq B$ and $\mathcal{V}\cap(\p^2\times\Lambda)\subset \C^2\times\Lambda$. We let
\[\mathcal{W}_B:=\{\mathcal{V}\in \mathcal{Z}_B\, : \ \deg_\mathcal{M}(\mathcal{f}^{n}(\mathcal{V}))\leq B, \ \text{for all} \ n\in \Z \}.\]
The set $\mathcal{W}_B$ is an algebraic subvariety of $\mathcal{Z}_B$ which is stable under iteration of $\mathcal{f}$ and $\mathcal{f}^{-1}$ and which contains the grand orbit $\mathcal{O}(\mathcal{Z}):=\{\mathcal{f}^n(\mathcal{Z})\, : \ n\in \Z \}$ of $\mathcal{Z}$. Define
\[\widehat{\mathcal{W}}:=\{(\mathcal{V},(x,y),\lambda)\in \mathcal{W}\times \C^2\times\Lambda\, : \ ((x,y),\lambda)\in \mathcal{V}\}\]
and let $\pi_{\mathcal{W}}:\widehat{\mathcal{W}}\to\mathcal{X}_\Lambda$ be the projection onto the second factor. By assumption, any irreducible component of $\widehat{\mathcal{W}}$ containing $\mathcal{f}^n(\mathcal{Z})$ for some $n\in\Z$ has dimension at least $2$ (otherwise, $\mathcal{W}_B$ is finite and $\mathcal{Z}$ is periodic). Assume first $\widehat{\mathcal{W}}$ has an irreducible component of dimension $2$. Then  it must have a periodic component of dimension $2$, and for all $\lambda$ outside a finite subset of $\Lambda$, there exists an algebraic curve $C_\lambda$ of $\C^2$ such that $f_\lambda^m(C_\lambda)\subset C_\lambda$ for some integer $m$. This is a contradiction, since regular polynomial automorphisms have no invariant algebraic curve.
This implies any periodic irreducible component of $\widehat{\mathcal{W}}$ has dimension at least $3$ and the orbit $\mathcal{O}(\mathcal{C})$ of $\mathcal{C}$ is Zariski dense in $\p^2\times\mathcal{B}$. Let $\widehat{\mathcal{W}}_1$ be such a component. Up to replacing $f$ by an iterate, we can assume $\mathcal{f}(\widehat{\mathcal{W}}_1)=\widehat{\mathcal{W}}_1$.

\begin{lemma}\label{dim=3}
The canonical projection $\Pi:\widehat{\mathcal{W}}_1\to\C^2\times\Lambda$ is an isomorphism and, for any $\lambda\in\Lambda$, the restriction $\Pi_\lambda:\Pi^{-1}(\C^2\times \{\lambda\})\to\C^2$ of $\Pi$ is an isomorphism.
\end{lemma}

We take Lemma \ref{dim=3} for granted for now and finish the proof. Fix $\lambda_0\in\Lambda$. For any $\lambda\in\Lambda$, we thus can define a birational map $\phi_\lambda\in\mathrm{Bir}(\p^2)$ by letting
\[\phi_\lambda=\Pi_\lambda\circ\Pi_{\lambda_0}^{-1}.\]
By construction, for any $\lambda\in\Lambda$, the indeterminacy points of $\phi_\lambda$ are contained in the line at infinity, we have $f_\lambda=\phi_\lambda\circ f_{\lambda_0}\circ\phi_\lambda^{-1}$ and, for any $\mathcal{V}\in \mathcal{W}_1$ and any $\lambda\in \Lambda$, we have
\[\phi_\lambda(V_{\lambda_0})=V_\lambda.\]
We now rely on Proposition \ref{prop:global-stable}.
By construction of the map $\phi_\lambda$, this implies
\[\phi_{\lambda}(G_{\lambda_0}^\pm= \alpha)= 
(G_{\lambda}^\pm= \alpha),\]
for any $\alpha\geq0$.
Assume, by contradiction, that the map $\phi_{\lambda}$ does not extend as a holomorphic map on $\p^2$, then it contracts the line at infinity $L_\infty$. In particular, for a point $(x,y)\in L_\infty\setminus I(\phi_{\lambda})$, we have $\phi_\lambda(x,y)\in I(\phi_{\lambda}^{-1})$. In other word, a small neighborhood $U_+$ of $(x,y)$ should be sent into a small neighborhood $U_-$ of an indeterminacy point of $\phi_{\lambda}^{-1}$. This is impossible, since any $(x',y')\in U_+$ satisfies $G_{\lambda_0}^+(x',y')=\alpha^+$ and $G_{\lambda_0}^+(x',y')=\alpha^-$ for some $\alpha^\pm>0$ very large. Then, $G_{\lambda}^\pm(\phi_{\lambda}(x',y'))=\alpha^\pm$ and this is not close to  the indeterminacy point of $\phi_{\lambda}^{-1}$.
This is a contradiction so $\phi_\lambda$ is in fact a biholomorphism and the family is isotrivial. We have proved that if $(\Lambda,\mathcal{f},\mathcal{Z})$ and $(\Lambda,\mathcal{f}^{-1},\mathcal{Z})$ are stable, then $\mathcal{f}^n(\mathcal{Z})=\mathcal{Z}$ for some $n\in\mathbb{Z}^*$.
 
~

To conclude the proof, it is sufficient to show that if the pair $(\Lambda,\mathcal{f},\mathcal{Z})$ is stable, then the pair $(\Lambda,\mathcal{f}^{-1},\mathcal{Z})$ is also stable. By assumption, there exists $M>0$ such that $\deg_\mathcal{M} \mathcal{f}^n(\mathcal{Z})\leq M$ for all $n$. Consider the variety $\mathcal{W}^+:=\{V\in \mathcal{C}_M\, : \ \deg(\mathcal{f}^{n}(V))\leq M, \ \text{for all} \ n \geq 0 \}$.
The result follows if we show that $\mathcal{W}^+$ is a finite set. Assume to the contrary that $\dim\mathcal{W}^+>0$. As before, we may replace $\mathcal{W}^+$ by some of its irreducible components and $\mathcal{f}$ by $\mathcal{f}^n$ to have $\mathcal{f}(\mathcal{W}^+)=\mathcal{W}^+$. Define 
\[\widetilde{\mathcal{W}^+}:=\{(\mathcal{V},(x,y),\lambda)\in \mathcal{W}^+\times  \C^{2}\times \Lambda \, : \ ((x,y),\lambda)\in \mathcal{V}\}\]
and let $\pi_{\mathcal{W}^+}:\widetilde{\mathcal{W}^+}\to\C^{2}\times\Lambda$ be the canonical projection. Then either $\dim \widetilde{\mathcal{W}^+}=2$ or $\dim \widetilde{\mathcal{W}^+}\geq3$. Again, if $\dim\widetilde{\mathcal{W}^+}=2$,  for all $\lambda$ outside a finite subset of $\Lambda$, there exists an algebraic curve $C_\lambda$ of $\C^2$ such that $f_\lambda(C_\lambda)\subset C_\lambda$, which is impossible. So we can assume that $\dim\widetilde{\mathcal{W}^+}\geq 3$. In particular, the map $\pi_{\mathcal{W}^+}:\widetilde{\mathcal{W}^+}\to\C^{2}\times\Lambda$ is dominant. Let $((x_0,y_0),\lambda_0)\in \pi_{\mathcal{W}^+}(\widetilde{\mathcal{W}^+})$ be such that $(x_0,y_0)$ is a saddle periodic point of $f_{\lambda_0}$ and let $\mathcal{Z}'\in\mathcal{W}^+$ be such that $((x_0,y_0),\lambda_0)\in\mathcal{Z}'$ and let $z'$ be the restriction of $\mathcal{Z}'$ to the generic fiber of $\pi:\p^2\times\Lambda\to\Lambda$. We have $\widehat{h}_f^+(z')=0$, whence
\[\widehat{h}_f(f^n(z'))=\widehat{h}_f^-(f^n(z'))=d^{-n}\widehat{h}_f^-(z')\longrightarrow_{n\to+\infty}0.\]
As the sequence $(\mathcal{f}^n(\mathcal{Z}'))_n$ has bounded global volume, by Bishop's Theorem, there exists a subsequence $(\mathcal{f}^{n_k}(\mathcal{Z}'))_k$ which converges to a curve $\mathcal{Z}''$ with $\deg_\mathcal{M}(\mathcal{Z}'')\leq M$. Moreover, if $z''$ is the generic fiber of the curve $\mathcal{Z}''$, the above implies
\[\widehat{h}_f(z'')=0.\]
By the previous step of the proof, $\mathcal{Z}''$ belongs to the finite set $\mathcal{W}$. 
This is a contradiction, since there are infinitely many distinct saddle periodic points.
\end{proof}

\begin{remark}\normalfont
\begin{enumerate}
\item  This proof shares many similarities with the proof of the Northcott property over number field of \cite[Theorem 4.2]{Kawaguchi-Henon}. Indeed, in both cases, we first establish a finiteness property for the total height $\widehat{h}_f$, then when $z$ only satisfies $\widehat{h}^+_f(z)=0$, we push it forward to produce point of small height $\widehat{h}_f$ and we can conclude using this finiteness property for $\widehat{h}_f$.
 \item In fact, we have proved that for a given $B>0$, the set of points $z\in\mathbb{A}^2(\mathbf{K})$ with $\deg_\mathcal{M}(\mathcal{Z}_n)\leq B$ for all $n\in\mathbb{Z}$ is finite (here $\mathcal{Z}_n$ is the Zariski closure of $f^n(z)$ in $\mathbb{P}^2\times\mathcal{B}$).
\end{enumerate}
\end{remark}
We now conclude by the proof of Lemma \ref{dim=3}.

\begin{proof}[Proof of Lemma~\ref{dim=3}]
By construction, the map $\Pi$ is dominant. First, we use a similarity argument. Pick $((x_0,y_0),\lambda_0)$ such that $(x_0,y_0)$ is a saddle periodic point of $f_{\lambda_0}$ of period $q\geq1$ and let $\mathcal{V}\in\mathcal{W}$ with $((x_0,y_0),\lambda_0) \in \mathcal{V}$. Let $(\lambda, V_\lambda)$ denote a local parametrization of $\mathcal{V}$. Let also $p(\lambda)$ be the continuation of $(x_0,y_0)$ as a saddle periodic point of $f_\lambda$, for $\lambda\in U$.

Up to replacing $\mathcal{f}$ with $\mathcal{f}^q$ we can assume $q=1$ for this part of the proof. Take $\lambda$ in a small neighborhood $U$ of $\lambda_0$. We know that $V_\lambda$ belong to a small neighborhood of $p(\lambda)$ where the stable and unstable manifold of $p(\lambda)$ intersects transversely. Assume, by symmetry, that $V_\lambda$ does not belong to the stable manifold of $p(\lambda)$. Considering $f^n_\lambda(V_\lambda)$ for $n\gg 1$, we then deduce that the family $\lambda \mapsto (\lambda, f^n_\lambda(V_\lambda))$ is not normal, a contradiction. In particular, $V_\lambda$ belong to both the stable and unstable manifold of $p(\lambda)$ so it equal to $p(\lambda)$ in a neighborhood of $\lambda_0$:
$\mathcal{V}\cap U=\{(p(\lambda),\lambda)\, : \ \lambda\in U\}$.

~

As the set of points of the form $((x_0,y_0), \lambda_0)$ such that $(x_0,y_0)$ is a saddle periodic point of $f_{\lambda_0}$ is Zariski dense in $\C^2\times\Lambda$, the map $\Pi$ is a birational morphism whose image is a Zariski open subset of $\C^2\times\Lambda$. Let $H$ be its complement. Up to removing a finite set from $\Lambda$, the set $H\cap  (\C^2  \times \{\lambda\})$ is Zariski closed. Moreover, it is both $f_\lambda$ and $f_\lambda^{-1}$ invariant. Since $f_\lambda$ has no invariant curve, this forces $H\cap  (\C^2  \times \{\lambda\})$ to be finite so this is a finite union of periodic orbit. Take any sequence $\mathcal{V}_n$ such that, for a given $\lambda$, the set $\mathcal{V}_n\cap (\C^2  \times \{\lambda\})$ accumulates on $H\cap (\C^2  \times \{\lambda\})$. By Bishop theorem, up to extraction, $(\mathcal{V}_n)_n$ converges towards $\mathcal{V}_\infty\subset H$ with $\deg(\mathcal{V}_\infty)\leq D$. This implies $H$ is a finite union of curves $\mathcal{V}\in \mathcal{W}$. This contradicts the definition of $H$, so that $\Pi$ is surjective.

We now need to prove it is finite. Since $\C^2\times\Lambda$ is normal, this would end the proof that it is an isomorphism. As in the case of endomorphisms, if non-empty, the set
$\{((x,y),\lambda)\, : \ \mathrm{Card}(\Pi^{-1}\{(x,y),\lambda\})=+\infty\}$ is totally invariant by $\mathcal{f}$, whence consists of periodic points. As they are isolated, this is impossible and $\Pi$ is finite.

\medskip

Fix now $\lambda_0\in\Lambda$. Note that, since $\mathcal{C}$ is a graph, it is irreducible and we have
\[\left(\mathcal{C}\cdot (\C^2  \times \{\lambda\})\right)=1\]
for all $\lambda\in\Lambda$, where $\C^2  \times \{\lambda\}$ is a general fiber of $\pi$. As all the fibers of $\pi$ (resp. all varieties in $\mathcal{W}$) are cohomologous and as the intersection can be computed in cohomology, we deduce that
\[\left(\mathcal{V}\cdot (\C^2  \times \{\lambda\})\right)=1\]
for all $\lambda\in\Lambda$ and all $\mathcal{V}\in\mathcal{W}$. In particular, $\Pi_\lambda:\Pi^{-1}(\C^2  \times \{\lambda\})\to\C^2$ is a surjective morphism whose topological degree is exactly $\left(\mathcal{V}\cdot (\C^2  \times \{\lambda\})\right)$ for a general variety $\mathcal{V}\in\mathcal{W}$, the map $\Pi_\lambda$ is a surjective finite birational morphism. By normality of $\C^2$, it is an isomorphism.
\end{proof}

\medskip

We are now in position to prove our main result:

\begin{proof}[Proof of the Main Theorem]
We first prove item $(1)$. The equivalence between the four points in question is an immediate consequence of Proposition~\ref{prop:global-stable}, Proposition~\ref{lm:height-henon} and Theorem~\ref{tm:henon}.

The finiteness of periodic points is an immediate consequence of Proposition~\ref{lm:height-henon}.
\end{proof}

\subsection{On a conjecture of Kawaguchi and Silverman for regular automorphisms}

As an application, we prove  the conjecture of Kawaguchi and Silverman \cite[Conjecture 6]{Kawaguchi-Silverman-arithmetic-degree}  in the case of regular polynomial automorphisms defined over a function field of characteristic zero (due to Kawaguchi and Silverman over number fields).

Recall that the first \emph{dynamical degree} $\lambda_1(f)$ of a regular polynomial automorphism $f$ is
\[\lambda_1(f):=\lim_{n\to+\infty}\left((f^n)^*H\cdot H\right)^{1/n},\]
where $H$ is any big and nef Cartier divisor on $\mathbb{A}^2$, and that the \emph{arithmetic degree} of a point $P\in \mathbb{A}^2(\bar{\mathbf{K}})$ is  
\[\alpha_f(P):=\lim_{n\to+\infty}\max\left(h(f^n(P)),1\right)^{1/n},\]
where $h=h_{\p^2,\mathcal{O}(1)}$ as above, when the limit exists. Actually, by Proposition~\ref{Stable=bounded-henon}, we have that for any point $P\in \mathbb{A}^2(\bar{\mathbf{K}})$, $\alpha_f(P)$ exists and is equal to $1$ or $d =\lambda_1(f)$. In particular, points 1, 2 and 3 of the conjecture hold.

Now, as an immediate consequence of Theorem~\ref{tm:henon} and Proposition~\ref{Stable=bounded-henon}, either $P$ has finite orbit, or $\widehat{h}_f(P)>0$ and $\alpha_f(P)=d$. This actually strengthens point 4 of the conjecture, where this is supposed to hold for points of Zariski dense orbit.

We thus have proven the following

\begin{corollary}
Let $k$ be a field of characteristic zero and $\mathcal{B}$ be a smooth projective $k$-curve. Let $f$ be a non-isotrivial regular polynomial automorphism over $\mathbf{K}:=k(\mathcal{B})$ of degree $d$.
\begin{enumerate}
\item  For any $P\in \mathbb{A}^2(\bar{\mathbf{K}})$, the limit $\alpha_f(P)$ exists and is an integer,
\item the set $\{\alpha_f(Q)\, : \ Q\in X(\bar{\mathbf{K}})\}$ coincides with $\{1,d\}=\{1,\lambda_1(f)\}$,
\item if $P\in \mathbb{A}^2(\bar{\mathbf{K}})$ has infinite forward (or backward) orbit, we have $\alpha_f(P)=\lambda_1(f)$.
\end{enumerate}
\end{corollary}

\end{document}